\documentclass{article}
\usepackage[utf8]{inputenc}
\usepackage[english]{babel}
\usepackage{blindtext}
\usepackage{amssymb}
\usepackage{amsthm}
\usepackage{amsmath}
\usepackage{tikz-cd}
\usepackage{setspace}
\usepackage{thmtools}
\usepackage{thm-restate}
\usepackage{hyperref}
\usepackage{cleveref}
\usepackage[margin=1.2in]{geometry}

\usepackage{graphicx}
\graphicspath{ {images/} }

\theoremstyle{plain}
\newtheorem{theorem}{Theorem}[section]

\theoremstyle{definition}
\newtheorem{definition}[theorem]{Definition}

\theoremstyle{plain}
\newtheorem{prop}[theorem]{Proposition}

\theoremstyle{plain}
\newtheorem{corollary}[theorem]{Corollary}

\theoremstyle{plain}
\newtheorem{lemma}[theorem]{Lemma}

\theoremstyle{remark}
\newtheorem{remark}[theorem]{Remark}

\theoremstyle{remark}
\newtheorem{example}[theorem]{Example}

\theoremstyle{plain}

\theoremstyle{plain}
\newtheorem{question}[theorem]{Question}

\title{Limiting characters at ideal points detecting twice-punctured tori}
\author{Yi Wang}
\date{}

\begin{document}

\maketitle

\begin{abstract}
The limiting character, introduced by Tillmann \cite{tillmann}, has been studied recently in the context of Culler-Shalen theory. We extend the methods of the author's previous work \cite{Wang} to show that certain families of essential twice-punctured tori are detected by an ideal point on the character variety and determine the limiting character at these ideal points. We then provide numerous explicit examples, including certain two-bridge knots, 3-strand pretzel knots, and knots with non-integral toroidal surgeries. We also prove that the union of a once- and a twice-punctured torus inside the $(-3, 5, 5)$ or $(3, -5, -5)$ pretzel knot, both essential, is detected by an ideal point of the character variety and explicitly determine its limiting character. 
\end{abstract}

\section{Introduction}\label{sec:intro}

Given a one-cusped hyperbolic 3-manifold $M$, the \emph{$SL_2(\mathbb{C})$-character variety} $X(M)$ is an associated algebraic set. The character variety is a source of arithmetic invariants of 3-manifolds and provides a bridge between low-dimensional topology and arithmetic geometry. In particular, \emph{Culler-Shalen theory} \cite{Culler1983VarietiesOG} associates ideal points of $C$ to essential surfaces in $M$.

\begin{definition}
Let $C \subset X(M)$ be a one-dimensional irreducible component of the character variety, and let $x$ be an ideal point of $C$. An essential surface $S \subset M$ is said to be \emph{detected by $x$} if $S$ is an essential surface associated to $x$ via the process described in \cite{Culler1983VarietiesOG}. 
\end{definition}

This paper is primarily interested in determining the \emph{limiting character} at ideal points detecting essential surfaces. 

\begin{definition}
Let $x \in C$ be an ideal point, and let $S \subset M$ be the detected surface. Given a connected component $M_i \subset M \setminus S$, the \emph{limiting character at $M_i$} is 
\begin{equation}
	\chi_i^\infty = \lim_{\chi \to x}\chi|_{\pi_1(M_i)}
\end{equation}
It is a consequence of the work in \cite{Culler1983VarietiesOG} that $\chi_i^\infty(g)$ takes finite values for all $g \in \pi_1(M_i)$. 
\end{definition}

The limiting character was first alluded to in the ``roots of unity" phenomenon of Cooper, Gillet, et al \cite{Cooper1994PlaneCA}. When investigating the root of unity phenomenon, Dunfield \cite{Dunfield2008AP1} described the concept of a limiting character. The notion was formalized by Tillmann \cite{tillmann}, and limiting characters at ideal points play a role in the detection of Conway spheres by Paoluzzi-Porti \cite{Paoluzzi2010ConwaySA}. In \cite{Wang}, the author proved a result about limiting characters of essential surfaces detecting once-punctured tori. To state this result, we first define some terminology which will be used throughout the paper. 

\begin{definition}
Given a closed 3-manifold $M$, the \emph{JSJ decomposition} is a (unique up to isotopy) disjoint union of essential tori whose complementary regions are either hyperbolic or Seifert-fibered. The \emph{JSJ graph} of a JSJ decomposition is a graph whose vertices are complementary regions and whose edges are JSJ tori. 
\end{definition}

\begin{definition}
Given a hyperbolic 3-manifold $M$, the \emph{holonomy representation} $\rho: \pi_1(M) \to PSL_2(\mathbb{C})$ is the canonical map into $PSL_2(\mathbb{C})$ given by the hyperbolic structure (see Chapter 2 of \cite{Thurston1979TheGA}). Given a 3-manifold $M$ Seifert-fibred over a hyperbolic 2-orbifold $\mathcal{O}$, the \emph{holonomy representation} $\rho: \pi_1(M) \to PSL_2(\mathbb{R}) \hookrightarrow PSL_2(\mathbb{R})$ is $[\pm I]$ on the regular fiber, and reduces to the holonomy representation of $\mathcal{O}$ on the 2-orbifold. In an abuse of notation, we will also use the term \emph{holonomy representation} to denote lifts $\tilde{\rho}: \pi_1(M) \to SL_2(\mathbb{C})$, which exist due to Culler \cite{Culler1986LiftingRT}. The \emph{holonomy trace} will be the character of $\tilde{\rho}$. 
\end{definition}

\begin{remark}
In this paper, we will also be in the situation where we need a ``holonomy representation" for the \emph{twisted $I$-bundle over the Klein bottle}, a 3-manifold with torus boundary Seifert-fibered over a Euclidean orbifold. In Section \ref{sec:background}, we will discuss what it means to be a ``holonomy representation" for this space.
\end{remark}

Let $M$ be an integral homology solid torus (i.e. a 3-manifold with torus boundary which has the same homology as a solid torus). 



\begin{definition}
A \emph{system of punctured JSJ tori} $\{T_i\}_{i=1}^n$ is a disjoint union of non-isotopic essential $p_i$-punctured tori $T_i$, satisfying the property that for some slope $\beta$ of the torus boundary, $M(\beta)$ admits a JSJ decomposition with exactly $n$ essential tori, all of which come from capping of $T_i$ by $p_i$ points in the Dehn filling. A \emph{system of $p$-punctured JSJ tori} is a system of punctured JSJ tori such that all $p_i = p$. 
\end{definition}

\begin{definition}
Given a slope $\beta$ on the torus boundary, we say that $M$ is \emph{$\beta$-rigid} if $X(M(\beta))$ is 0-dimensional. If $\beta = 0$, we say that $M$ is \emph{longitudinally rigid}. 
\end{definition}

\begin{theorem}\cite{Wang}
Let $M = S^3 \setminus K$ be a longitudinally rigid non-fibered hyperbolic knot complement, and let $\{T_i\}_{i=1}^n \subset M$ be a system of once-punctured JSJ tori with slope 0. Then $\bigcup_{i=1}^nT_i$ is detected by an ideal point $x$ on an irreducible component of $X(M)$ for which the trace of the longitude is nonconstant, and the limiting character at $x$ is the trace of the holonomy representation of the JSJ complementary regions of $M(0)$. 
\end{theorem}

In other words, given a set of genus one Seifert surfaces that cap off to a JSJ decomposition, there exists an ideal point in the character variety detecting that (possibly disconnected) surface, and we gain information about the limiting character. 

\begin{remark}
This theorem was motivated by the following. Chinburg-Reid-Stover \cite{Chinburg2017AzumayaAA} defined a particular hyperbolic knot invariant called the \emph{canonical Azumaya algebra} $A_{k(C)}$, associated to the function field $k(C)$ of an irreducible component $C$ of $X(M)$. A $K$-theoretic argument was used in \cite{Chinburg2017AzumayaAA} to extend equivalence classes of Azumaya algebras over \emph{ideal points}, i.e. ``points at infinity", of $C$. In \cite{Wang}, the author initiated the study of \emph{tautological extension} of $A_{k(C)}$ over ideal points of $C$, which is a natural refinement of the Chinburg-Reid-Stover invariant. It turns out that if the limiting character at an ideal point is irreducible, tautological extension is satisfied at that ideal point. This has no bearing on the content of this paper, but is discussed in detail in the background sections of \cite{Wang}. 
\end{remark}

\begin{remark}
The above theorem is stated differently from its statement in \cite{Wang}, since the original statement was more in terms of the Chinburg-Reid-Stover invariant. In this restatement, the hypotheses are modified to better fit the context and motivations of this paper, and the exact same proof goes through. 
\end{remark}

The following corollary establishes an infinite family of examples to which the above theorem applies:

\begin{corollary}\label{cor:twobridgeseifert}
Any genus one Seifert surface in a genus one nonfibered two-bridge knot complement $M$ is detected by an ideal point of a component $C$ in $X(M)$ where the limiting character is the holonomy trace of the JSJ component of $M(0)$. 
\end{corollary}

In this paper, we provide an alternate proof of the ideal point detection of almost all twice-punctured JSJ tori in two-bridge knots and investigate the phenomenon in some pretzel knots. In addition, the limiting characters at these ideal points are completely determined. The main theorem is as follows: 

\begin{theorem}\label{thm:main}
Suppose $M = S^3 \setminus K$ is an $\beta$-rigid hyperbolic knot complement, where $\beta$ is the boundary slope of a system of separating twice-punctured JSJ tori $\mathbb{T} = \bigcup_{i=1}^{n-1}T_i$ satisfying the following: 
\begin{enumerate}
	\item The JSJ graph of $M(\beta)$ is a line 
	\item $M(\beta)$ has JSJ components that are either 
	\begin{enumerate}
		\item Seifert-fibered over $D(p, q)$ with $p \geq 2, q > 2$
		\item Seifert-fibered over $A^2(r)$ with $r \geq 2$
		\item hyperbolic
		\item the twisted $I$-bundle over the Klein bottle
	\end{enumerate}
	\item If $n = 2$, the JSJ torus in $M(\beta)$ does not bound two twisted $I$-bundles of Klein bottles
	\item Let $\mathcal{O}$ be a Seifert-fibered JSJ component $\mathcal{O}$ of $M(\beta)$ meeting a twisted $I$-bundle over the Klein bottle, denoted $\tilde{K}$. Then the regular fiber of $\mathcal{O}$ is not glued to the regular fiber of $\tilde{K}$ Seifert-fibered over $D^2(2, 2)$
\end{enumerate}
Then $\mathbb{T}$ is detected by an ideal point $x$ on an irreducible component of the character variety for which the trace of $\beta$ is nonconstant, and the limiting character at $x$ is the holonomy trace of the JSJ complementary regions of $M(\beta)$. 
\end{theorem}

The central ideal of the proof is to mimic the arguments of Paoluzzi-Porti \cite{Paoluzzi2010ConwaySA}, which shows an analogous theorem for essential Conway spheres. In particular, we show in Section \ref{sec:background} that holonomy representations of the JSJ components of $M(\beta)$ have the same traces on the boundary tori, but they are nonconjugate. This means that there is a sequence of characters approaching the holonomy representations of the JSJ components of $M(\beta)$, but the limit cannot glue to a legitimate representation of the fundamental group of the original knot complement. This is captured in Lemma 7 of \cite{Paoluzzi2010ConwaySA}, which we use to conclude our theorem. 

\medskip

We then demonstrate that this theorem applies to several explicit well-known families of knots, such as two-bridge knots and particular pretzel knots. One notable family of knots to which Theorem \ref{thm:main} applies is the \emph{Eudave-Mu$\tilde{\text{n}}$oz knots}, which were shown by \cite{eudavemunoz} and \cite{gordonleucke} to be the only knot complements in $S^3$ which admitted non-integral toroidal Dehn fillings. Theorem \ref{thm:main}, along with Corollary \ref{cor:twobridgeseifert}, is used to show the following, more explicit corollary.

\begin{corollary}\label{cor:examples}
Let $\mathbb{T} \subset M$ be one of the following systems of punctured JSJ tori in knot complements:
\begin{enumerate}
	\item any such system in a two-bridge knot that is not the figure-eight knot or the trefoil
	\item an essential twice-punctured torus in a $(-2, 3, 2n+1)$-pretzel knot with $n \not\equiv 1 \text{ (mod } 3$)
	\item an essential twice-punctured torus with slope $\beta$ in the Eudave-Mu$\tilde{\text{n}}$oz knots which are $\beta$-rigid
\end{enumerate} 
Then $\mathbb{T}$ is detected by an ideal point $x$ on a curve in the character variety for which the trace of $\beta$ is nonconstant, and the limiting character is the holonomy trace of the JSJ complementary regions of $M(\beta)$. 
\end{corollary}

\begin{remark}\label{rmk:known}
The detection results for all the knots in Corollary \ref{cor:examples} follows quickly using previous results on slope detection. Detection of punctured tori in two-bridge knots follows from combining the results of Hatcher-Thurston \cite{Hatcher1985IncompressibleSI} and Ohtsuki \cite{ohtsuki}. Detection of punctured tori in $(-2, 3, 2n+1)$ pretzel knots comes from combining the work of Mattman \cite{Mattman} and Hatcher-Oertel \cite{hatcheroertel}. Detection of twice-punctured tori in Eudave-Mu$\tilde{\text{n}}$oz knots does not seem to be proven in the literature, though the author suspects that arguments from Ni-Zhang \cite{nizhang} may be used to conclude this. The main value of the results in this paper is to establish the limiting character, and to give geometric meaning to the detection of certain essential surfaces through ideal points. The limiting character itself has some applications, namely to describe extensions of the Chinburg-Reid-Stover invariant from \cite{Chinburg2017AzumayaAA} and potential applications to ordering 3-manifold groups, as in the author's other paper \cite{orderability}. 
\end{remark}

\begin{remark}
The hypotheses of longitudinal or $\beta$-rigidity are stronger than what is needed for the proof; a more technical and specific condition suffices, and this is used to demonstrated detection results for Seifert surfaces of $(-3, 3, 2n+1)$ pretzel knots, which are not longitudinally rigid, in the author's other paper \cite{orderability}. For cases where we do not have longitudinal or $\beta$-rigidity, more detailed analysis of the topology of $M$ and the character variety is required to complete analogous detection / limiting character results. Such an example is described below, and explicitly proven in the final section of this paper. 
\end{remark}

We also consider systems of punctured JSJ tori consisting of once- and twice-punctured tori. The first known examples come from a pair of odd pretzel knots. Sekino \cite{sekino} determined that if $M$ is the $(-3, 5, 5)$ or the $(3, -5, -5)$ pretzel knot complement, there exists a once- and a twice-punctured torus with boundary slope 0, and the JSJ decomposition of $M(0)$ has two components, one of which is a thickened thrice-punctured sphere, and the other is the trefoil knot complement. We have the following theorem, in vein with the previous results:

\begin{theorem}
Let $M$ be the $(-3, 5, 5)$ or $(3, -5, -5)$ pretzel knot complement. Then the system of punctured JSJ tori at slope 0 is detected by an ideal point on the character variety on which the trace of the longitude is nonconstant, and the limiting character at that ideal point restricts to the holonomy representation of the thrice-punctured sphere and the trefoil knot complement. 
\end{theorem}

\subsection{Outline of the paper}

Section \ref{sec:background} covers necessary background for the proof of Theorem \ref{thm:main}. Section \ref{sec:proof} goes through the proof of Theorem \ref{thm:main}. Section \ref{sec:examples} applies Theorem \ref{thm:main} to several families of knots, including those addressed in Corollary \ref{cor:examples}. 

\subsection{Acknowledgements}

The author thanks the anonymous referees for providing valuable comments on an earlier draft of the paper, particularly alternative arguments for detection of ideal points. The author also thanks Nathan Dunfield for providing data on Eudave-Mu$\tilde{\text{n}}$oz knots. 

\section{Background for the proofs}\label{sec:background}

\subsection{The JSJ structure of $M(\beta)$}

Let $M = S^3 \setminus K$ be a hyperbolic knot complement containing a system of twice-punctured JSJ tori $\mathbb{T} = \bigcup_{i=1}^{n-1}T_i$ with slope $\beta$, and the JSJ graph of $M(\beta)$ is a line. Since each $T_i$ is separating, there are $n$ complementary regions $M_i \subset M \setminus \mathbb{T}$. The case where $n = 3$ is depicted below. Here $M_1$ and $M_3$ have boundary a genus 2 surface, while $M_2$ has boundary a genus 3 surface. 

\begin{figure}[h]
	\centering
	\includegraphics[scale=.5]{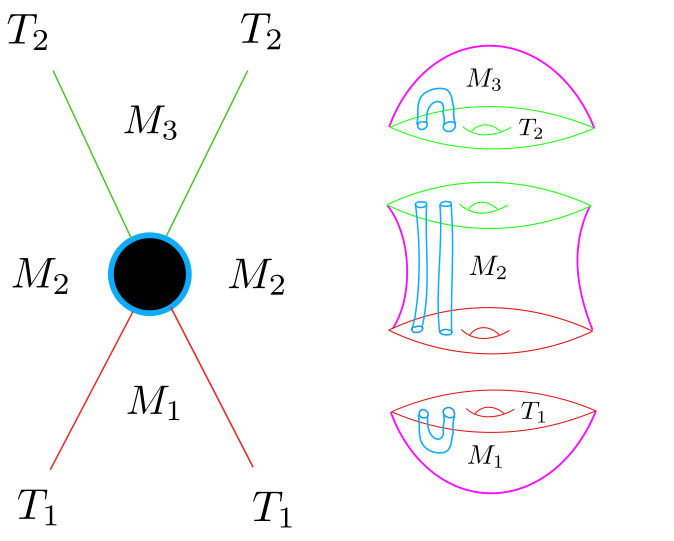}
	\caption{Left: Schematic of how $T_1$ and $T_2$ are arranged in $M$. Right: The complementary regions $M_i$ when cutting $M$ via $T_1$ and $T_2$. Note the blue annuli, which come from a tubular neighborhood of the knot.}
\end{figure}

Under Dehn filling, each blue annulus in the above figure is filled with a 2-handle. Under this handle addition, $M_i$ becomes $\mathcal{O}_i$, a JSJ complementary region of $M(\beta)$. In particular, the complementary regions at the two end vertices have one torus boundary component, while the rest of the complementary regions have two torus boundary components. In summary, we have the following lemma:


\begin{lemma}\label{lma:jsjstructure}
Given a system of twice-punctured JSJ tori $\mathbb{T} = \bigcup_{i=1}^{n-1}T_i$ with slope $\beta$ such that the JSJ graph of $M(\beta)$ is a line, there are $n$ complementary regions $M_i \subset M$, two of which have boundary a single genus 2 surface, and the rest of which have $\partial M_i$ a genus 3 surface. In $M(\beta)$, there are $n - 1$ JSJ tori $\widetilde{T_i}$ with $n$ complementary regions $\mathcal{O}_i$, two of which have one torus as their boundary, and the rest of which have two tori as their boundary. 
\end{lemma}

\begin{remark}
This is a different situation compared to the once-punctured torus case in \cite{Wang}, in which the JSJ graph of $M(0)$ graph was a circle. This slightly changes the calculations necessary for the proof of Theorem \ref{thm:main}.
\end{remark}

\subsection{The twisted $I$-bundle over the Klein bottle}

One space that often appears in JSJ decompositions of Dehn fillings capping of twice-punctured tori in knot complements is the \emph{twisted $I$-bundle over the Klein bottle}, denoted $\tilde{K}$. This is an orientable 3-manifold with torus boundary. For an in-depth discussion of $\tilde{K}$, we refer the reader to \cite{LidmanWatson}. 

\medskip

We will study representations of $\pi_1(\tilde{K})$. It is known that
\begin{equation}
	\pi_1(\tilde{K}) = \langle s, t \mid s^2t^2 = 1 \rangle
\end{equation}
The incompressible torus boundary has fundamental group generated by the meridian $st$ and longitude $s^2$. It will be useful to compute the representation and character varieties of this group. Utilizing the trace relation
\begin{equation}
	\text{tr}(AB) = \text{tr}(A)\text{tr}(B) - \text{tr}(AB^{-1})
\end{equation}
we can determine that the character variety is a complex cubic surface:
\begin{equation}
	X(\tilde{K}) = \left\{(x, y, z) \in \mathbb{C}^3 \mid xyz - x^2 - y^2 = 0\right\}
\end{equation}
with coordinates are $x = \text{tr}(s), y = \text{tr}(t), z = \text{tr}(st)$. We now determine all the irreducible $SL_2(\mathbb{C})$-representations of $\pi_1(\tilde{K})$. Any such representation is conjugate to
\begin{equation}
	\rho(s) = \begin{pmatrix}x&1\\0&x^{-1}\end{pmatrix} \ \ \ \ \ \rho(t) = \begin{pmatrix}y&0\\r&y^{-1}\end{pmatrix}
\end{equation}
Determining $\rho(s^2t^2)$ and setting the entries equal to the identity, we obtain that any irreducible representation must satisfy
\begin{equation}
	\rho(s) = \begin{pmatrix}\pm i&1\\0&\mp i \end{pmatrix} \ \ \ \ \ \rho(t) = \begin{pmatrix}\pm i&0\\r&\mp i\end{pmatrix} \ \ \ \ \ \rho(st) = \begin{pmatrix}r-1&-i\\-ri&-1\end{pmatrix}
\end{equation}
which correspond to the point $(0, 0, r-2) \in X(\tilde{K})$. The reducible representations can be conjugated into
\begin{equation}
	\rho(s) = \begin{pmatrix}x&1\\0&x^{-1}\end{pmatrix} \ \ \ \ \ \rho(t) = \begin{pmatrix}y&r\\0&y^{-1}\end{pmatrix}
\end{equation}
Once again solving for $\rho(s^2t^2) = I$, we find that the reducible representations can be given by two families. 
\begin{enumerate}
	\item The representations that satisfy\begin{equation}
		\rho(s) = \begin{pmatrix}\pm i&1\\0&\mp i\end{pmatrix} \ \ \ \ \ \rho(t) = \begin{pmatrix}\pm i&r\\0&\mp i\end{pmatrix}
	\end{equation}
	which correspond to the points $(0, 0, \pm 2) \in X(\tilde{K})$. 
	\item The representations that satisfy \begin{equation}
		\rho(s) = \begin{pmatrix}x&1\\0&x^{-1}\end{pmatrix} \ \ \ \ \ \rho(t) = \begin{pmatrix}\pm x^{-1}&\mp1\\0&\pm x\end{pmatrix}
	\end{equation}
	which correspond to the points $(a, \pm a, \pm 2) \in X(\tilde{K})$. Note that the traces of these points contain the traces of the previous type of reducible representation.
\end{enumerate}
We will be interested in representations of $\pi_1(\tilde{K})$ where $\text{tr}(st) = z = \pm 2$ and $\text{tr}(s^2) = x^2 - 2 = \pm 2$. From the character variety, we see that the only points in $X(\tilde{K})$ satisfying this condition are $(0, 0, \pm 2)$ or $(\pm2, \pm2, \pm2)$ (with an even number of negatives) i.e. the representations 
\begin{equation}
	\rho_A(s) = \begin{pmatrix}\pm i&1\\0&\mp i\end{pmatrix} \ \ \ \ \ \rho_A(t) = \begin{pmatrix}\pm i&0\\4&\mp i\end{pmatrix}
\end{equation}
\begin{equation}
	\rho_B(s) = \begin{pmatrix}\pm i&1\\0&\mp i\end{pmatrix} \ \ \ \ \ \rho_B(t) = \begin{pmatrix}\pm i&r\\0&\mp i\end{pmatrix}
\end{equation}
\begin{equation}
	\rho_C(s) = \begin{pmatrix}\pm1&1\\0&\pm1\end{pmatrix} \ \ \ \ \ \rho_C(t) = \begin{pmatrix}\pm1&\mp1\\0&\pm 1\end{pmatrix}
\end{equation}
Notice that all of these representations are reducible, since the coordinates always satisfy $x^2 + y^2 + z^2 - xyz - 4 = 0$. We will use this in the proof of the main theorem. To summarize, we state the following lemma:

\begin{lemma}\label{lma:twistedbundle}
Let $\tilde{K}$ be the twisted $I$-bundle over the Klein bottle. Then the following are the representations of $\pi_1(\tilde{K})$ such that $\text{tr}(st) = \pm 2$ and $\text{tr}(s^2) = \pm 2$:
\begin{equation}
	\rho_A(s) = \begin{pmatrix}\pm i&1\\0&\mp i\end{pmatrix} \ \ \ \ \ \rho_A(t) = \begin{pmatrix}\pm i&0\\4&\mp i\end{pmatrix}
\end{equation}
\begin{equation}
	\rho_B(s) = \begin{pmatrix}\pm i&1\\0&\mp i\end{pmatrix} \ \ \ \ \ \rho_B(t) = \begin{pmatrix}\pm i&r\\0&\mp i\end{pmatrix}
\end{equation}
\begin{equation}
	\rho_C(s) = \begin{pmatrix}\pm1&1\\0&\pm1\end{pmatrix} \ \ \ \ \ \rho_C(t) = \begin{pmatrix}\pm1&\mp1\\0&\pm 1\end{pmatrix}
\end{equation}
\end{lemma}

\subsection{$SL_2(\mathbb{C})$-compatibility}

Suppose $M = S^3 \setminus K$ is a hyperbolic knot complement, with $\mathbb{T}$ a system of twice-punctured JSJ tori with slope $\beta$. We must analyze the behavior of the signs of the traces of the boundary components of $M(\beta)$. In particular, the proof of the main theorem requires a choice of a set of characters on the JSJ complementary regions $M(\beta)$ such that the traces on the boundary tori match. This property, defined in \cite{Wang}, is known as \emph{$SL_2(\mathbb{C})$-compatibility}. Proving $SL_2(\mathbb{C})$-compatibility makes heavy use of \emph{negative-flexibility}, which we define here: 

\begin{definition}
A pair $(M, \rho)$ consisting of a 3-manifold $M$ with torus boundary components $T_1, \dots, T_m$ and a representation $\rho: \pi_1(M) \to PSL_2(\mathbb{C})$ is called \emph{negative-flexible} if for any $(\gamma_1, \dots, \gamma_m) \in \pi_1(T_1) \times \dots \times \pi_1(T_m)$ with $\gamma_i$ simple and nontrivial, there exists a lift $\hat{\rho}: \pi_1(M) \to SL_2(\mathbb{C})$ such that $\text{tr}(\hat{\rho}(\gamma_i)) = -2$ for all $i = 1, \dots, m$. 
\end{definition}

To begin, we recall the following lemma, proved in \cite{Wang}:

\begin{lemma}\label{lma:hyperbolicandannulus}
Suppose $M$ is either hyperbolic or Seifert-fibered over $A^2(q)$. Let 
\begin{equation}
	\rho = \begin{cases}
		\text{holonomy reprsentation of } M & $M$ \text{ hyperbolic} \\ \text{holonomy representation of } A^2(q) & $M$ \text{ Seifert-fibered over } A^2(q)
	\end{cases}
\end{equation}
Then $(M, \rho)$ is negative-flexible. 
\end{lemma}

We now prove similar lemmas for two more types of spaces.

\begin{lemma}\label{lma:disknegative}
Suppose $M$ is Seifert-fibered over $D(p, q)$ with $p \geq 2, q > 2$. Let $\rho$ be the holonomy representation of $D(p, q)$. Then $(M, \rho)$ is negative-flexible.
\end{lemma}

\begin{proof}
The holonomy representation of the thrice-punctured sphere, whose fundamental group is given by $\langle a, b, c \mid c = ab \rangle$, is 
\begin{equation}
	\rho(a) = \left[\pm\begin{pmatrix}1 & 2 \\ 0 & 1\end{pmatrix}\right] \ \ \ \ \ \rho(b) = \left[\pm\begin{pmatrix}1&0\\-2&1\end{pmatrix}\right] \ \ \ \ \ \rho(c) = \left[\pm\begin{pmatrix}-3&2\\-2&1\end{pmatrix}\right]
\end{equation}
In particular, for any lift $\hat{\rho}$ of this representation to $SL_2(\mathbb{C})$, if the signs of $\hat{\rho}(a)$ and $\hat{\rho}(b)$ are the same, then the sign of $\hat{\rho}(c)$ is negative, and if the traces of $\hat{\rho}(a)$ and $\hat{\rho}(b)$ are different, then the sign of $\hat{\rho}(c)$ is positive. The holonomy representation of $D^2(p, q)$, $p \geq 2, q > 2$ is thus given by 
\begin{equation}
	\rho(a) = \left[\pm \begin{pmatrix}1 & 2 \\ 0 & 1\end{pmatrix}\right] \ \ \ \ \ \rho(b) = \left[\pm\begin{pmatrix}\xi_{2p}&0\\x_q&\xi_{2p}^{-1}\end{pmatrix}\right] \ \ \ \ \ \rho(c) = \left[\pm\begin{pmatrix}\xi_{2p} + 2x_q & 2\xi_{2p}^{-1} \\ x_q & \xi_{2p}^{-1}\end{pmatrix}\right]
\end{equation}
where
\begin{equation}
	\xi_{2p} + \xi_{2p}^{-1} + 2x_q = \xi_{2q} + \xi_{2q}^{-1}
\end{equation}
Since the holonomy of $D^2(p, q)$ is a continuous deformation of the holonomy of the thrice-punctured sphere, it follows that the signs of the traces of a lift of $a$ and $b$ are equal if and only if $c$ lifts to a negative trace, as is true for the thrice-punctured sphere. We will show that in all cases, the $\mathbb{Z}/2\mathbb{Z}$-homologically trivial element of $\pi_1(\partial M)$, i.e. the one with trace sign fixed in all lifts, must lift to a negative trace. Let $h$ be the regular fiber. 

\medskip

\underline{Case 1: Either $p$ or $q$ are even.} Without loss of generality, let $q$ be even. In this case, $c^qh^r = 1$ for $r$ odd. Since $q$ is even, given any lift $\hat{\rho}: \pi_1(M) \to SL_2(\mathbb{C})$, $\text{tr}(\hat{\rho}(c^q)) = -2$. By Lemma 3.1 in \cite{Kitano1994ReidemeisterTO}, for all lifts, $\hat{\rho}$, $\hat{\rho}(h) = \pm I$. Since $r$ is odd and $c^qh^r = 1$, this means that $\text{tr}(\hat{\rho}(h)) = -2$. Since $\text{tr}(\hat{\rho}(h)) = -2$ for all lifts $\hat{\rho}$, this means that $[h] \in H^1(M; \mathbb{Z}/2\mathbb{Z})$ is trivial. By the half-lives, half-dies theorem, $[h]$ must generate the kernel of the induced map $H^1(\partial M; \mathbb{Z}/2\mathbb{Z}) \to H^1(M; \mathbb{Z}/2\mathbb{Z})$. Since $\text{tr}(\hat{\rho}(h)) = -2$, this shows that $(M, \rho)$ is negative-flexible in this case.

\medskip

\underline{Case 2: $p$ and $q$ are both odd.} In this case, $b^ph^{r_1} = c^qh^{r_2} = 1$. 

\smallskip

\underline{Subcase 2.1}: $r_1$ and $r_2$ are both even. In this case, for all lifts $\hat{\rho}$, $\hat{\rho}(h^{r_i}) = I$, $i = 1, 2$. This means that $\text{tr}(\hat{\rho}(b^p)) = \text{tr}(\hat{\rho}(c^q)) = 2$, and since $p$ and $q$ are both odd, this means that $\text{tr}(\hat{\rho}(b)) = -(\xi_{2p} + \xi_{2p}^{-1})$ and $\text{tr}(\hat{\rho}(c)) = -(\xi_{2q} + \xi_{2q}^{-1})$. Since $c$ lifts to a negative trace, the signs of the traces of $a$ and $b$ are equal. Thus $\text{tr}(\hat{\rho}(a)) = -2$ for all lifts $\hat{\rho}$, and as in the previous case, $[a]$ generates the kernel of the boundary-induced $\mathbb{Z}/2\mathbb{Z}$-homology map. Thus $(M, \rho)$ is negative-flexible. 

\smallskip

\underline{Subcase 2.2}: $r_1$ is even and $r_2$ is odd. Then for all lifts $\hat{\rho}$, $\hat{\rho}(h^{r_1}) = 1$, and so $\text{tr}(\hat{\rho}(b^p)) = 2$, hence $\text{tr}(\hat{\rho}(b)) = -(\xi_{2p} + \xi_{2p}^{-1})$. In addition, if $\text{tr}(\hat{\rho}(c)) = -(\xi_{2q} + \xi_{2q}^{-1})$, then $\text{tr}(\hat{\rho}(c^q)) = 2$ and thus $\text{tr}(\hat{\rho}(h^{r_1})) = \text{tr}(\hat{\rho}(h)) = 2$. If $\text{tr}(\hat{\rho}(c)) = \xi_{2q} + \xi_{2q}^{-1}$, then $\text{tr}(\hat{\rho}(c^q)) = -2$ and $\text{tr}(\hat{\rho}(h^{r_1})) = \text{tr}(\hat{\rho}(h)) = -2$, i.e. $c$ and $h$ have opposite sign traces under every lift. In particular, if $\text{tr}(\hat{\rho}(c))$ is negative, then $\text{tr}(\hat{\rho}(a))$ and $\text{tr}(\hat{\rho}(b))$ have the same sign. It follows that $\text{tr}(\hat{\rho}(a))$ is negative while $\text{tr}(\hat{\rho}(h))$ is positive. If $\text{tr}(\hat{\rho}(c))$ is positive , then $\text{tr}(\hat{\rho}(a))$ and $\text{tr}(\hat{\rho}(b))$ have opposite signs. It follows that $\text{tr}(\hat{\rho}(a))$ is positive and $\text{tr}(\hat{\rho}(h))$ is negative. In either case, $a$ and $h$ have negative traces, so $ah$ always has a negative trace. Thus $[ah]$ generates the kernel of the boundary-induced $\mathbb{Z}/2\mathbb{Z}$-homology map. Thus $(M, \rho)$ is negative-flexible. 

\smallskip

\underline{Subcase 2.3}: $r_1$ is odd and $r_2$ is even. A symmetric argument to the previous case shows that $(M, \rho)$ is also negative-flexible in this case. 

\smallskip

\underline{Subcase 2.4}: $r_1$ and $r_2$ are odd. Then by the argument involving $\text{tr}(\hat{\rho}(c))$ in Subcase 2.2, both $\text{tr}(\hat{\rho}(c))$ and $\text{tr}(\hat{\rho}(b))$ have the opposite sign trace as $\text{tr}(\hat{\rho}(h))$; this means that $\text{tr}(\hat{\rho}(c))$ and $\text{tr}(\hat{\rho}(b))$ have the same sign. If $\text{tr}(\hat{\rho}(c))$ is negative, then the traces of $a$ and $b$ under the same lift are equal. But since $\text{tr}(\hat{\rho}(b))$ is also negative, $\text{tr}(\hat{\rho}(a))$ is negative as well. If $\text{tr}(\hat{\rho}(c))$ is positive, then the traces of $a$ and $b$ under that lift are different. But since $\text{tr}(\hat{\rho}(b)$ is also positive, it follows that $\text{tr}(\hat{\rho}(a))$ is once again negative. Thus $a$ has negative trace for all lifts, and jence $[a]$ generates the kernel of the boundary-induced $\mathbb{Z}/2\mathbb{Z}$-homology map. Thus $(M, \rho)$ is negative-flexible. 
\end{proof}

\begin{remark}
In light of Lemmas \ref{lma:disknegative} and \ref{lma:hyperbolicandannulus}, a reasonable conjecture is the following. For any 3-manifold $M$ Seifert-fibered over a hyperbolic 2-orbifold $\mathcal{O}$, $p: M \to \mathcal{O}$ the fibering map, and $\rho$ the holonomy representation of $\mathcal{O}$, $(M, \rho \circ p)$ is negative-flexible. A proof sketch mimics the proof for negative-flexibility of hyperbolic 3-manifolds in Lemma 3.9 of \cite{MenalFerrer2010TwistedCF}, though there are many details left to be filled in. 
\end{remark}

\begin{lemma}\label{lma:ibundlenegative}
Suppose $M$ is the twisted $I$-bundle over the Klein bottle $\tilde{K}$. Let $\rho_A, \rho_B, \rho_C: \pi_1(\tilde{K})$ be as in Lemma \ref{lma:twistedbundle}. Then $(M, \rho_A)$ and $(M, \rho_B)$ are negative-flexible, while $(M, \rho_C)$ is not negative-flexible. 
\end{lemma}

\begin{proof}
Recall that 
\begin{equation}
	\rho_A(s) = \begin{pmatrix}\pm i&1\\0&\mp i\end{pmatrix} \ \ \ \ \ \rho_A(t) = \begin{pmatrix}\pm i&0\\4&\mp i\end{pmatrix}
\end{equation}
\begin{equation}
	\rho_B(s) = \begin{pmatrix}\pm i&1\\0&\mp i\end{pmatrix} \ \ \ \ \ \rho_B(t) = \begin{pmatrix}\pm i&r\\0&\mp i\end{pmatrix}
\end{equation}
\begin{equation}
	\rho_C(s) = \begin{pmatrix}\pm1&1\\0&\pm1\end{pmatrix} \ \ \ \ \ \rho_C(t) = \begin{pmatrix}\mp1&\pm1\\0&\mp 1\end{pmatrix}
\end{equation}
The behavior of the representations on the boundary torus is as follows:
\begin{equation}
	\rho_A(st) = \begin{pmatrix}3&\mp i\\\mp 4i&-1\end{pmatrix} \ \ \ \ \ \rho_A(s^2) = \begin{pmatrix}-1&0\\0&-1\end{pmatrix}
\end{equation}
\begin{equation}
	\rho_B(st) = \pm\begin{pmatrix}-1&\pm i(r-1)\\0&-1\end{pmatrix} \ \ \ \ \ \rho_B(s^2) = \begin{pmatrix}-1&0\\0&-1\end{pmatrix}
\end{equation}
\begin{equation}
	\rho_C(st) = \begin{pmatrix}\pm1&0\\0&\pm1\end{pmatrix} \ \ \ \ \ \rho_C(s^2) = \begin{pmatrix}1&\pm 2\\0&1\end{pmatrix}
\end{equation}
Notice that $\rho_A(s^2)$ and $\rho_B(s^2)$ are fixed at $-I$, while $\rho_C(s^2)$ has trace 2. This proves the lemma. 
\end{proof}

The proof of the following theorem largely mimics the proof of Theorem 3.2 in \cite{Wang}.

\begin{theorem}\label{thm:compatible}
Let $M$ be a closed irreducible 3-manifold with JSJ decomposition with JSJ tori $\{T_i\}_{i=1}^{n-1}$, gluing homomorphisms $\{\varphi_i: \partial M_i \to \partial M_{i+1}\}_{i=1}^{n-1}$, and complementary regions $\{M_i\}_{i=1}^n$, such that:
\begin{enumerate}
	\item The JSJ graph is a line, with endpoints $M_1, M_n$. 
	\item $M_1, M_n$ are either Seifert-fibered over $D(p, q)$ with $p \geq 2, q > 2$ or the twisted $I$-bundle over the Klein bottle.
	\item The other $M_i$ are either hyperbolic or cable spaces.
	\item If $n = 2$, then $M_1$ and $M_2$ cannot both be the twisted $I$-bundle over the Klein bottle.
\end{enumerate}
Then the JSJ decomposition of $M$ is $SL_2(\mathbb{C})$-compatible. Moreover, there exist $SL_2(\mathbb{C})$-compatible choices $\rho_i$ of representations on the $\pi_1(M_i)$ such that if $M_1$ or $M_n$ is the twisted $I$-bundle over the Klein bottle, $\rho_i$ is one of $\rho_A$ or $\rho_B$ from Lemma \ref{lma:twistedbundle}. 
\end{theorem}

\begin{proof}
Take any $M_i$ that is not the twisted $I$-bundle over the Klein bottle, and let $\rho_i$ be the associated hyperbolic holonomy representation. If $M_i$ is the twisted $I$-bundle of the Klein bottle, then let $\rho_i$ be either $\rho_A$ or $\rho_B$ from Lemma \ref{lma:twistedbundle}. By Lemmas \ref{lma:hyperbolicandannulus}, \ref{lma:disknegative}, and \ref{lma:ibundlenegative}, all the $(M_i, \rho_i)$ are negative-flexible. 

\medskip

We know that the $M_i$, $i \neq 1, n$ (i.e. the JSJ components with two torus boundaries) can be categorized into two types: Type 1 and Type 3 (see \cite{Wang}, proof of Theorem 3.2). By the ``half-lives, half-dies" theorem (Lemma 5.3 in \cite{Schleimer2018IntroductionTT}), $\iota: \partial M_i \to M_i$ is the boundary inclusion, then $\ker(\iota_*)$ is half-dimensional in $H_1(\partial M; \mathbb{Z}/2\mathbb{Z})$. In the case where $\partial M_i$ has two torus boundary components, this means that $\ker(\iota_*) \cong (\mathbb{Z}/2\mathbb{Z})^2$ as a subspace of $(\mathbb{Z}/2\mathbb{Z})^4$. A \emph{Type 1} JSJ component is one where $\ker(\iota_*)$ is generated by two homology classes in separate torus boundary components, and a \emph{Type 3} JSJ component is one where $\ker(\iota_*)$ is generated by two homology classes both of which include homology classes from both tori. 

\medskip

For $M_i$, $i = 1, n$, by the half-lives half-dies theorem and negative-flexibility, the generator of the kernel of the induced homology map must have trace -2 for any lift $\widehat{\rho_i}$ or $\rho_i$ to $SL_2(\mathbb{C})$. We now have the following possibilities for the $M_i, i \neq 1, n$:
\begin{itemize}
	\item There are no such $M_i$, i.e. $n = 2$. In this case, two spaces with one torus boundary component are glued together. Let $\pi_1(\partial M_i) = \langle m_i, \ell_i \rangle$ where $\ell_i$ is the $\mathbb{Z}/2\mathbb{Z}$-homologically trivial element of $\pi_1(M_i)$. If $\varphi_1: \partial M_1 \to \partial M_2$ takes $[\ell_1]$ to $[\ell_2]$, then we simply make sure the other negatives match. If not, then the two $\mathbb{Z}/2\mathbb{Z}$-homology classes will be locked into trace -2. This realizes $SL_2(\mathbb{C})$-compatibility. 
	\item All the $M_i$ are of Type 1. Repeat the above argument for each $\varphi_i$. This will realize $SL_2(\mathbb{C})$-compatibility. 
	\item All the $M_i$ are of Type 3. Note that both $M_1$ and $M_n$ have $\text{tr}(\widehat{\rho_i}(\ell_i)) = -2$ for $i = 1, 2$, while the traces of the $\mathbb{Z}/2\mathbb{Z}$-homology class of $m$ and $m\ell$ must have the opposite sign under any lift. For each type 3 component, let the homology class glued to $\ell_1$ and $\ell_n$ have trace -2 under the chosen lift. This will satisfy the $SL_2(\mathbb{C})$-compatibility condition. 
	\item The $M_i$ are mixed type. Note that type 3 components can be ignored in this setting, so we are in the same situation as the case where the $M_i$ are of Type 1. So the same argument suffices. 
\end{itemize}
\end{proof}

\begin{remark}
From here on, a ``holonomy representation" for the twisted $I$-bundle over the Klein bottle is given by $\rho_A$ or $\rho_B$.
\end{remark}

\section{Proof of main theorem}\label{sec:proof}

We now prove Theorem \ref{thm:main}.

\begin{proof}[Proof of Theorem \ref{thm:main}]
	Let $M_i$ denote the complementary regions of $\overline{M \setminus \mathbb{T}}$ for $i = 1, \dots, n$, and let $\widehat{\mathbb{T}} = \bigcup_{i=1}^{n-1}\hat{T_i}$ denote the JSJ tori in $M(\beta)$. We have the restriction map 
	\begin{equation}
		r: X(M) \to \mathcal{X} = \prod X(M_i)
	\end{equation}
	By Lemma \ref{lma:jsjstructure}, $M_1, M_n$ have a single genus 2 surface as their boundaries, while any other complementary regions $M_i$ (if they exist) have a single genus 3 surface as their boundaries. Note that by Theorem 5.6 in \cite{Thurston1979TheGA}, $\dim(X(M_i)) \geq 3$ for $i = 1, n$, and $\dim(X(M_i)) \geq 6$ for any other $M_i$. Thus $\dim(\mathcal{X}) \geq 6(n-1)$. Let $V$ be the variety inside $\mathcal{X}$ defined by the following equations.
	\begin{enumerate}
		\item In viewing $M_i$ as the complementary regions of the knot complement $M$, there are $n - 1$ gluing maps between twice-punctured tori in $M_i$. For each gluing, there are seven matching equations. If the first twice-punctured torus has $\pi_1$ a free group generated by $a_1, b_1, c_1$ and the other one is a free group generated $a_2, b_2, c_2$, the equations are
		\begin{equation}
			\text{tr}(a_1) = \text{tr}(a_2) \ \ \ \ \ \text{tr}(b_1) = \text{tr}(b_2) \ \ \ \ \ \text{tr}(c_1) = \text{tr}(c_2) \ \ \ \ \ \text{tr}(a_1b_1) = \text{tr}(a_2b_2)
		\end{equation}
		\begin{equation}
			\text{tr}(a_1c_1) = \text{tr}(a_2c_2) \ \ \ \ \ \text{tr}(b_1c_1) = \text{tr}(b_2c_2) \ \ \ \ \ \text{tr}(a_1b_1c_1) = \text{tr}(a_2b_2c_2) \ \ \ \ \ \text{tr}(a_1c_1b_1) = \text{tr}(a_2c_2b_2)
		\end{equation}
		We make the following variable assignments.
		\begin{equation}
			\text{tr}(a_1) = x_1 \ \ \ \ \ \text{tr}(b_1) = x_2 \ \ \ \ \ \text{tr}(c_1) = x_3 \ \ \ \ \ \text{tr}(a_1b_1) = x_{12}
		\end{equation}
		\begin{equation}
			\text{tr}(a_1c_1) = x_{13} \ \ \ \ \ \text{tr}(b_1c_1) = x_{23} \ \ \ \ \ \text{tr}(a_1b_1c_1) = x_{123} \ \ \ \ \ \text{tr}(a_1c_1b_1) = x_{132}
		\end{equation}
		Similarly, we assign $\text{tr}(a_2) = x_1'$, and so on. We now show that only five of these equations suffice to define $V$. Suppose we have the equations
		\begin{equation}
			x_1 = x_1' \ \ \ \ \ x_2 = x_2' \ \ \ \ \ x_3 = x_3' \ \ \ \ \ x_{12} = x_{12}' \ \ \ \ \ x_{13} = x_{13}'
		\end{equation} 
		On each twice-punctured torus, since the punctures correspond to the same loop $[\beta]$ in $M$, the two punctures are known to be have equal trace. We thus have the trace equation
		\begin{equation}
			\text{tr}(bc^{-1}) = \text{tr}(aba^{-1}c^{-1})
		\end{equation} 
		which translates into the equations
		\begin{equation}
			x_2x_3 - x_3 = x_{12}x_{13} - x_1x_{123} + x_{23} \ \ \ \ \ x_2'x_3' - x_3' = x_{12}'x_{13}' - x_1'x_{123}' + x_{23}'
		\end{equation}
		Hence we can use our existing six equations to conclude that $x_{23} = x_{23}'$. By Section 5.1 in \cite{Goldman}, $(x_{123}, x_{132})$ and $(x_{123}', x_{132}')$ are pairs of solutions to quadratics with coefficients in the other six coordinates, which we already know are equal. Thus, either $(x_{123}, x_{132}) = (x_{132}', x_{123}')$ or $(x_{123}, x_{132}) = (x_{123}', x_{132}')$; each option represents one component of the variety $V$. Pick the component with $(x_{123}, x_{132}) = (x_{123}', x_{132}')$; we now have an irreducible component defined with $5(n - 1)$ equations (5 equations for $n - 1$ gluings). 
		\item The other equations defining $V$ are as follows: there are $n - 2$ complementary regions $M_i$ that have two twice-punctured tori inside the two genus 2 surface boundary components $S_1, S_2$. For each such component, there is one extra equation which dictates that the trace of the gluing annulus (i.e. the annulus coming from the knot complement) on $S_1$ is equal to the corresponding trace on $S_2$.
	\end{enumerate}
	Thus $V$ is defined by $(5n - 5) + (n - 2) = 6n - 7$ equations in $\mathcal{X}$, whose dimension is at least $6n - 6$. Hence $\dim(V) \geq 1$. We now show that $V$ is in fact one-dimensional. Notice that $\text{Im}(r) \subset V$, since any restriction of a character in $X(M)$ satisfies the gluing equations, and let $\gamma$ be the product of two punctures in $\mathbb{T}$, hence a commutator in $\pi_1(\mathbb{T})$. Let $V^{irr} \subset V$ be the Zariski-open subset of $V$ such of traces which restrict to irreducible traces on $\pi_1$. This contains the Zariski-open set where $\text{tr}(\gamma) \neq 2$. This means that $V^{irr}$ is a subset of traces that match on the boundary and are irreducible on $T$, which glue to give traces of $\pi_1(M)$ by Lemma 6 in \cite{Paoluzzi2010ConwaySA}. Hence $V^{irr} \subset \text{Im}(r)$. By assumption, $\dim(X(M(\beta))) = 0$, and so $\dim(X(M)) = 1$, and hence $\dim(\text{Im}(r)) \leq 1$. The dimension of $V^{irr}$ cannot be zero, since this would mean that $\dim(V^{irr}) < \dim(V)$ and it is Zariski open in $V$, hence $V^{irr}$ is empty, contradicting the fact that $M$ is a hyperbolic knot complement and so has a positive-dimensional character variety. So $\dim(V^{irr}) = \dim(V)$. But $\dim(V^{irr}) \leq 1$, and $\dim(V) \geq 1$, so they must both be equal to 1. This also means that $\dim(X(M_1)) = \dim(X(M_n)) = 3$ and for all other $M_i$, $\dim(X(M_i)) = 6$.
	
	\medskip
	
	Let $\mathcal{O}_i$ denote the JSJ complementary regions of $M(\beta)$ corresponding to $M_i$ with a 2-handle glued to its boundary. By Theorem \ref{thm:compatible}, $M(\beta)$ is $SL_2(\mathbb{C})$-compatible. Let $\chi_i$ be the corresponding traces realizing $SL_2(\mathbb{C})$-compatibility, and let $\chi_\infty = (\chi_1, \dots, \chi_n) \in V \setminus V^{irr}$. Since $M$ is $\beta$-rigid, $\text{tr}(\beta)$ cannot be constant near $\chi_\infty$. Any sequence approaching $\chi_\infty$ either has $\text{tr}(\gamma) = 2$ constantly on that sequence, in which case $\text{tr}(\beta) \neq 2$ and the trace of the commutator of the two punctures is not equal to 2, or $\text{tr}(\gamma) \neq 2$. In either case, the trace of a commutator is not constantly equal to 2 on the sequence, and so we must have a sequence of points $\{\chi_j\}_{j=1}^\infty \subset V^{irr}$ approaching $\chi_\infty$. Let $\{\alpha_j\} \in X(M)$ be such that $r(\alpha_j) = \chi_j$. Suppose for contradiction that up to subsequence, $\{\alpha_j\}$ converges to a character $\alpha_\infty \in X(M)$. Then $\alpha_\infty$ must be the trace of some $\rho_\infty: \pi_1(M) \to SL_2(\mathbb{C})$. Let $q_i: \pi_1(M_i) \to \pi_1(\mathcal{O}_i)$ be the quotient map induced by the gluing 2-handle. Since $r(\alpha_\infty) = \chi_\infty$, it follows that $\rho_\infty|_{\pi_1(M_i)}$ has the same trace as $\rho_i \circ q_i$. In the case where $\mathcal{O}_i$ is not the twisted $I$-bundle over the Klein bottle, this means that $\rho_i \circ q_i$ is irreducible, so $\rho_\infty|_{\pi_1(M_i)}$ is conjugate to $\rho_i \circ q_i$. In particular, $\rho_i|_{T_i} \circ q_i|_{\partial M_i}$ is the restriction of the holonomy representation to the boundary tori. When $\mathcal{O}_i$ is the twisted $I$-bundle over the Klein bottle, by the $SL_2(\mathbb{C})$-compatibility conditions from Theorem \ref{thm:compatible}, this means that $\rho_i|_{T_i} \circ q_i|_{\partial M_i} = \rho_j$ for $j = A, B$. We have the following cases for the incompressible twice-punctured torus $T_k \subset M$ that caps off to $\widehat{T_k} \subset M(\beta)$. Let $\widehat{T_{k_1}}, \widehat{T_{k_2}}$ be the corresponding boundary tori in the complementary JSJ components, with $\pi_1(\widehat{T_{k_1}}) = \langle m_1, \ell_1 \rangle, \pi_1(\widehat{T_{k_2}}) = \langle m_2, \ell_2 \rangle$.
	\begin{enumerate}
		\item $\widehat{T_k}$ bounds two Seifert-fibered spaces. In this case, $\ell_1, \ell_2$ are the regular fibers of the complementary regions, which have holonomy regions $\rho_1', \rho_2'$. Then up to conjugacy, 
		\begin{equation}
			\rho_k'(m_k) = \pm\begin{pmatrix}1&1\\0&1 \end{pmatrix} \ \ \ \ \ \rho_k'(\ell_k') = \pm\begin{pmatrix}1&0\\0&1\end{pmatrix}
		\end{equation}
		for $k = 1, 2$ and a choice of sign determined by $SL_2(\mathbb{C})$-compatibility. By Proposition 1.6.2 of \cite{Aschenbrenner20123ManifoldG}, the regular fibers don't match, i.e. if $\varphi: \langle m_1, \ell_2 \rangle \to \langle m_2, \ell_2 \rangle$ is the gluing map, then $\varphi(\ell_1) \neq \ell_2$. Thus, $\varphi^{-1}(\ell_2)$ is a nontrivial simple closed curve in $\pi_1(\widehat{T}_{k_1})$ that is not $\ell_1$; this means it must be of the form $m_1^p\ell_1^q$ with $p \neq 0$. However, 
		\begin{equation}
			\rho_1'(m_1^p\ell_1^q) = \pm \begin{pmatrix}1&p\\0&1\end{pmatrix} \neq \pm \begin{pmatrix}1&0\\0&1\end{pmatrix}
		\end{equation}
		which means that the two representations are not conjugate under the gluing map $\varphi$. This contradicts the assertion that $\alpha_\infty$ is the trace of some $\rho_\infty: \pi_1(M) \to SL_2(\mathbb{C})$. 
		\item $\widehat{T_k}$ bounds a Seifert-fibered and a hyperbolic space. The hyperbolic space will have a holonomy representation $\rho_h$ that restricts to the torus cusp $\langle m_h, \ell_h \rangle$ as follows. 
		\begin{equation}
			\rho_h(m_h) = \pm\begin{pmatrix}1&1\\0&1\end{pmatrix} \ \ \ \ \ \rho_h(\ell_h) = \pm\begin{pmatrix}1&\tau\\0&1\end{pmatrix}
		\end{equation}
		where $\text{Im}(\tau) > 0$, and the choice of sign is determined by $SL_2(\mathbb{C})$-compatibility. This is the \emph{cusp shape} of $\rho_h$. As seen previously, the cusp shape of the holonomy representation of a Seifert-fibered manifold is 0, but it is non-zero for a hyperbolic manifold, so the two torus representations are nonconjugate, leading to a contradiction.  
		\item $\widehat{T_k}$ bounds two hyperbolic spaces. Since the gluing map is orientation-reversing, the cusp shape of one of the holonomy representations has positive imaginary part, while the other has negative imaginary part. Since the cusp shapes are different, the holonomy representations restricted to the glued boundary tori cannot be conjugate, leading to a contradiction.  
		\item $\widehat{T_k}$ bounds a twisted $I$-bundle over the Klein bottle $\tilde{K}$ and a Seifert-fibered space. Let $\rho_{\tilde{K}}$ be $\rho_\infty$ restricted to $\tilde{K}$. We know that $\rho_{\tilde{K}} = \rho_A$ or $\rho_B$. In either case, the image of the regular fiber under the induced representation is negative the identity matrix, behaving the same way as a manifold Seifert-fibered over a hyperbolic orbifold. Thus, the argument from Case 1 leads to another contradiction. 
		\item $\widehat{T_k}$ bounds a twisted $I$-bundle over the Klein bottle and a hyperbolic component. Since no element of the holonomy of the hyperbolic component restricted to the boundary torus is plus or minus the identity matrix, that representation cannot be conjugate to the twisted Klein bottle holonomy restricted to the boundary torus. We thus have a contradiction.
	\end{enumerate}
	Thus, $\{\chi_j\}$ approaches an ideal point $x$ on a norm curve of $X(M)$, and the limiting character restricted the $M_i$ is $\chi_\infty$. By Lemma 7 in \cite{Paoluzzi2010ConwaySA}, $x$ detects $T$. Since the limiting character is irreducible for some $\mathcal{H}_i$, $A_{k(C)}$ tautologically extends over $x$, so we are done. 
\end{proof}

\section{Examples}\label{sec:examples}

The main family of examples on which we apply Theorem \ref{thm:main} are the punctured JSJ tori in alternating knots. The following theorem classifies such punctured JSJ tori:

\begin{theorem}\cite{IchiharaMasai}\label{thm:ichiharamasai}
Let $M = S^3 \setminus K$ be a hyperbolic alternating knot complement. Suppose $M(r)$ is toroidal but not Seifert-fibered. Then $K$ is equivalent to one of:
\begin{enumerate}
	\item the figure-eight knot, with $r = 0, \pm 4$
	\item a two-bridge knot $K_{(b_1, b_2)}$ with $|b_1|, |b_2| > 2$, with $r = 0$ if $b_1, b_2$ are even and $r = 2b_2$ if $b_1$ is odd and $b_2$ is even
	\item a twist knot $K_{(2n, \pm 2)}$ with $|n| > 1$ and $r = 0, \pm 4$
	\item a pretzel knot $P(q_1, q_2, q_3)$ with $q_i \neq 0, \pm 1$ for $i = 1, 2, 3$ and $r = 0$ if $q_i$ are all odd, and $r = 2(q_2 + q_3)$ if $q_1$ is even and $q_2, q_3$ are odd
\end{enumerate}
\end{theorem}

The cases where $K$ is a two-bridge knot and $r = 0$ are addressed in \cite{Wang}. In this section, we largely discuss the hyperbolic alternating knot complements admitting essential twice-punctured JSJ tori, i.e. the cases in the above theorem where $r \neq 0$. In order to apply Theorem \ref{thm:main} to these specific examples, we must understand the JSJ decompositions of the associated toroidal Dehn fillings.

\subsection{Two-bridge knots}

Let $M$ be a two-bridge knot complement. Theorem 1.2 of \cite{Macasieb2009OnCV} and Theorem 7.3 of \cite{Petersen2014GonalityAG} combine to show that all irreducible components of $X(M)$ are $\beta$-rigid for all slopes $\beta$. In order to satisfy the hypotheses of Theorem \ref{thm:main}, it only remains to study the JSJ decompositions of the toroidal Dehn fillings associated to the essential twice-punctured tori in two-bridge knots. For the case of the twice-punctured tori in twist knots, we use Proposition 2.2 from \cite{Teragaito}:

\begin{lemma}[\cite{Teragaito}]\label{lma:teragaito}
Let $M = S^3 \setminus K$ be a hyperbolic twist knot. Then $M(4, 1)$ is a graph manifold whose JSJ regions are the twisted $I$-bundle over the Klein bottle and a torus knot exterior. In addition, the regular fiber of the torus knot exterior is not identified with the regular fiber of the twisted $I$-bundle over the Klein bottle viewed as a Seifert-fibered space over $D^2(2, 2)$. 
\end{lemma}

For the case of the other two-bridge knots with twice-punctured tori, we use Lemma 3.1 from \cite{ClayTeragaito}:

\begin{lemma}[\cite{ClayTeragaito}]\label{lma:clayteragaito}
Let $M = S^3 \setminus K$ be the complement of the two bridge knot $K_{b_1, b_2}$ with $b_1$ odd and $b_2$ even. (Here $b_1, b_2$ coincides with the continud fraction $[-; b_1, b_2]$, as in \cite{Hatcher1985IncompressibleSI}.)Then $M(2b_2)$ is a graph manifold whose JSJ regions are the twisted $I$-bundle over the Klein bottle, a space Seifert-fibered over $A^2(b_2)$, and a torus knot exterior. In addition, the regular fiber of the space Seifert-fibered over $A^2(b_2)$ is not identified with the regular fiber of the twisted $I$-bundle over the Klein bottle viewed as a Seifert-fibered space over $D^2(2, 2)$. 
\end{lemma}

\begin{remark}
The lemma actually shows that the regular fiber of the cable space is identified with the regular fiber of $\tilde{K}$ viewed as a Seifert-fibered space over the Mobius strip. 
\end{remark}

These two results show that all two-bridge knots with essential twice-punctured JSJ tori in their knot complements satisfy the hypotheses of Theorem \ref{thm:main}. We thus have the following corollary. 

\begin{theorem}\label{thm:twobridge}
Any system of punctured JSJ tori in a two-bridge knot with boundary slope $\beta$, with the exception of the minimal genus Seifert surface of the figure-eight knot, is detected by an ideal point $x$ at which the limiting character $x$ restricts to the holonomy trace of the JSJ components described in Lemmas \ref{lma:teragaito} and \ref{lma:clayteragaito}.
\end{theorem}

As discussed in Remark \ref{rmk:known}, detection of these twice-punctured tori can be deduced from previous results. However, we now understand the limiting characters at some of these ideal points, giving the detection an additional geometric meaning. 

\subsection{Pretzel knots}

Less is generally known regarding character varieties and toroidal surgeries of pretzel knots. We have the following result from \cite{IchiharaJongKabaya} which determines the JSJ decomposition of the toroidal Dehn fillings of $(-2, p, q)$ pretzel knots with $3 \leq p \leq q$, $p, q$ odd:

\begin{theorem}[\cite{IchiharaJongKabaya}]
Consider the toroidal manifold $M$ obtained by $2(p+q)$-surgery on the hyperbolic $(-2, p, q)$ pretzel knot with odd integers $3 \leq p \leq q$. Then $M$ admits one JSJ torus which splits $M$ into the twisted $I$-bundle over the Klein bottle and $M_{p, q}$, which is either Seifert-fibered over $D^2(p, q)$, $p > 2, q \geq 2$, or hyperbolic. 
\end{theorem}

All of these pretzel knots satisfy the hypotheses of Theorem \ref{thm:main} involving the JSJ decompositions of the toroidal Dehn fillings. Furthermore, Mattman \cite{Mattman} showed the following:

\begin{theorem}[Theorem 1.6, Claim 6.3 in \cite{Mattman}]
The $SL_2(\mathbb{C})$-character variety of the $(-2, 3, 2n+1)$ pretzel knot complement with $n \ncong 1 \mod 3$ consists of the canonical component and the curve of reducible components. In addition, the regular fiber of the twisted $I$-bundle over the Klein bottle viewed as a fibering over $D^2(2, 2)$ does not match with the regular fiber of the other component.
\end{theorem}

These two results combined gives 

\begin{corollary}
Let $K$ be a $(-2, 3, 2n+1)$ pretzel knot, and let $M = S^3 \setminus K$ be its complement in the three-sphere. Then the twice-punctured torus of slope $2(2n+4)$ is detected by an ideal point $x$ on the canonical component of the character variety, and the limiting character restricts to the holonomy trace of $D^2(p, q)$ and the twisted $I$-bundle over the Klein bottle. 
\end{corollary}

This result combined with Theorem \ref{thm:twobridge} gives Corollary \ref{cor:examples}. As discussed in Remark \ref{rmk:known}, the detection result could already be deduced from previous results, but the limiting characters were not previously well-understood. 

\subsection{Eudave-Muñoz knots}

In \cite{Gordon2000DehnSO}, the following is proved.

\begin{theorem}[\cite{Gordon2000DehnSO}]
Let $K$ be a hyperbolic knot complement that admits a non-integral toroidal Dehn filling. Then $K$ admits one twice-punctured JSJ torus. 
\end{theorem}

In the context of this paper, it is natural to ask when twice-punctured tori associated to non-integral toroidal Dehn fillings are detected by ideal points in the character variety, and what its limiting character is. In order to answer this question, we describe a special family of knots called the \emph{Eudave-Muñoz knots} \cite{eudavemunoz}:

\begin{definition}
Let $(\ell, m, n, p)$ be four integers satisfying certain conditions. Construct a tangle $B(\ell, m, n, p)$ utilizing the 4-tuple of integers. The integers symbolize the number of twists in certain regions. For a picture of this tangle, see \cite{eudavemunoz}. A \emph{Eudave-Muñoz knot} $k(\ell, m, n, p)$ is the double-branched cover of $B(\ell, m, n, p)$.
\end{definition}

Eudave-Muñoz knots satisfy the following properties:

\begin{prop}\label{prop:eudavemunoz}\cite{eudavemunoz}
Let $K = k(\ell, m, n, p)$ be a Eudave-Muñoz knot. Then:
\begin{enumerate}
	\item $K$ admits a non-integral toroidal Dehn filling that induces a twice-punctured JSJ torus.
	\item Let $M$ be the non-integral toroidal surgery on $K$, with essential torus $T$. Then $T$ splits $M$ into $M_1$ and $M_2$, where $M_1$ and $M_2$ are Seifert-fibered over the disk with two exceptional fibers. 
	\item The regular fibers of $M_1$ and $M_2$ intersect once on $T$,
	\item One of the four exceptional fibers has multiplicity two. 
\end{enumerate}
\end{prop}

\begin{example}
The simplest example of a Eudave-Muñoz knot is the $(-2, 3, 7)$-pretzel knot, which is $k(3, 1, 1, 0)$. The non-integral toroidal Dehn filling is 37/2. 
\end{example}

In \cite{gordonleucke}, it was shown that any hyperbolic knot with a non-integral toroidal Dehn surgery is a Eudave-Muñoz knot. Since these non-integral toroidal Dehn fillings satisfy many of the hypotheses of Theorem \ref{thm:main}, except for $\beta$-rigidity. We thus have the following Corollary of Theorem \ref{thm:main}.

\begin{corollary}
Let $\beta(\ell, m, n, p)$ be the non-integral toroidal slope of $k(\ell, m, n, p)$. If the complement of $k(\ell, m, n, p)$ is $\beta$-rigid, then the twice-punctured JSJ torus in this knot complement is detected by an ideal point on the character variety, and the limiting character restricts to the holonomy traces of the $D^2(p, q)$ which form the bases of the JSJ components of the $\beta$-filling.
\end{corollary}

This begs the following question.

\begin{question}
Are all Eudave-Muñoz knot complements $\beta$-rigid, where $\beta$ is the non-integral toroidal slope?
\end{question}

\begin{remark}
In private communications with the author, Nathan Dunfield found all 83 non-integral toroidal Dehn fillings on hyperbolic knot complements that exist within the SnapPy census. The first few can be quickly computed to be $\beta$-rigid, where $\beta$ is the non-integral toroidal slope. Interestingly, all of the computed non-integral toroidal Dehn fillings are $L$-spaces, so by the results of \cite{hanselman}, the fundamental groups of these surgeries are non-left orderable. 
\end{remark}

\begin{question}
Are all non-integral toroidal Dehn fillings on Eudave-Muñoz knots $L$-spaces? Equivalently, are they left-orderable? 
\end{question}

\subsection{Odd pretzel knots}

In the case of odd pretzel knots, we have the following result of \cite{sekino}. 

\begin{theorem}[\cite{sekino}]
Let $M_{p, q, r}$ be the $(p, q, r)$ pretzel knot complement. We have the following descriptions of JSJ decompositions of $M_{2p+1, 2q+1, 2r+1}(0)$. 
\begin{itemize}
	\item $M_{-3, 3, 2n+1}(0)$ has one JSJ component which is the $(2, 4)$ torus link, which is Seifert fibered over the annulus with a cone point of order 2.
	\item $M_{-3, 5, 5}(0)$ and $M_{3, -5, -5}(0)$ has two JSJ components: the trefoil knot complement and the trivial circle fiber over the thrice-punctured sphere, denoted $S^1 \times S_{0, 3}$. 
	\item For any other $(p, q, r)$, $M_{2p+1, 2q+1, 2r+1}(0)$ has one JSJ component which is a hyperbolic 3-manifold. 
\end{itemize}
\end{theorem}

All odd pretzel knot complements except for $M_{-3, 5, 5}$ and $M_{3, -5, -5}$, have their Seifert surfaces detected by an ideal point by the results of \cite{Wang}. The limiting character information is used to prove left-orderability of fundamental groups of certain Dehn fillings of the $(-3, 3, 2n+1)$ pretzel knots in \cite{orderability}. In the case of the $M_{-3, 5, 5}$ and $M_{3, -5, -5}$ pretzel knots, there must be a twice-punctured torus with slope 0 that is disjoint from the Seifert surface. We have the following detection result which determines the limiting character at the ideal point. 

\begin{theorem}\label{thm:pretzels}
Let $M$ be the $(-3, 5, 5)$ or $(3, -5, -5)$ pretzel knot complement. Then the system of punctured JSJ tori at slope 0 is detected by an ideal point on the character variety on which the trace of the longitude is nonconstant, and the limiting character at that ideal point restricts to the holonomy representation of the thrice-punctured sphere and the trefoil knot complement. 
\end{theorem}

Let $M$ be the $(-3, 5, 5)$ or $(3, -5, -5)$ pretzel knot complement, and let $\mathbb{T} = T_1 \cup T_2 \subset M$ be the union of the essential once-punctured torus $T_1$ and twice-punctured torus $T_2$. 

\begin{lemma}\label{lma:pretzels}
$M_{-3, 5, 5}(0)$ and $M_{3, -5, -5}(0)$ are $SL_2(\mathbb{C})$-compatible.
\end{lemma}

\begin{proof}
Let $T_1, T_2, T_3$ be the three tori on $\mathcal{O}_1 = S^1 \times S_{0, 3}$, and $T$ be the torus on the trefoil knot complement, denoted $\mathcal{O}_2$. Say that $T_1$ is glued to $T$, while $T_2$ is glued to $T_3$. From the proof of Lemma \ref{lma:disknegative}, since the trefoil knot complement is Seifert-fibered over $D(2, 3)$, the regular fiber must be mapped to $-I$ under the holonomy representation. By the discussion in Section 3.1.4 of \cite{sekino}, the regular fiber in $T_1$ is glued to the meridian of $T$. Let $\rho_2: \pi_1(\mathcal{O}_2) \to SL_2(\mathbb{C})$ be the holonomy representation which takes the meridian to a matrix with trace 2, and let $\rho_1: \pi_1(\mathcal{O}_1) \to SL_2(\mathbb{C})$ be the representation which has trace 2 on all elements of $\pi_1(T_2)$ and $\pi_1(T_3)$ (so in particular on the regular fiber), and trace -2 on the meridian of $T_1$. This realizes $SL_2(\mathbb{C})$-compatibility. 
\end{proof}

We now mimic the proof strategy of Theorem \ref{thm:main}. At points where the proof of Theorem \ref{thm:main} applies \textit{mutatis mutandis}, we will refer to that proof.

\begin{proof}[Proof of Theorem \ref{thm:pretzels}]
Let $H = M \setminus T_1$ be the complement of just the Seifert surface; this is a genus 2 handlebody since the Seifert surface is built by Seifert's algorithm and is hence free. We thus have a map $r': X(M) \to X(H) = \mathbb{C}^3$. In order to compute the Zariski-closure of the image, we record the \emph{Lin presentation} of the fundamental group, which is found in \cite{sekino}:
\begin{equation}
	\pi_1(M) = \langle a, b, t \mid ta^3bababt^{-1} = a^3baba, tb^{-1}ababt^{-1} = b^{-1}ababa \rangle
\end{equation}
(Without loss of generality, we took only the fundamental group of $M_{-3, 5, 5}$. The proof for $M_{3, -5, -5}$ is exactly the same.) Here, $a, b$ are generators for the fundamental group of the Seifert surface, and $t$ is the meridian. We can compute the Zariski-closure of the image of $r$, denoted $V'$, as follows. Set 
\begin{equation}
	m_1 = a^3babab \ \ \ \ \ \ell_1 = a^3baba \ \ \ \ \ \ell_1 = b^{-1}abab \ \ \ \ \ \ell_2 = b^{-1}ababa
\end{equation}
Then let $x = \text{tr}(a), y = \text{tr}(b), z = \text{tr}(ab)$, and any word $w$ in $a, b$ can be expressed as a polynomial in $x, y, z$. The equations defining $V'$ are then
\begin{equation}
	\text{tr}(m_1) = \text{tr}(m_2) \ \ \ \ \ \text{tr}(\ell_1) = \text{tr}(\ell_2) \ \ \ \ \ \text{tr}(m_1\ell_1) = \text{tr}(m_2\ell_2)
\end{equation}
According to Sage, there are three one-dimensional components of this variety, defined by the following systems of equations:
\begin{equation}
	C_1 = \{x + z - 1 = z^2 + y - z - 2 = 0\}
\end{equation}
\begin{equation}
	C_2 = \{xz - y = x^2 + z^2 - 3 = z^3 + xy - 3z = 0\}
\end{equation}
\begin{equation}
	C_3 = \{x - z = yz^3 - y^2z - z^3 - yz - z^2 + y + 3z = 0\}
\end{equation}
Notice that $M \setminus \mathbb{T} = H \setminus T_2$ consists of two 3-manifolds, one of which has boundary a genus 2 surface, and the other of which has boundary a genus 3 surface. Denote the complementary regions $M_1, M_2$, where $\partial M_1 = S_1$, $\partial M_2 = S_2$, where $S_2$ is a genus 3 surface. 
\begin{figure}[h]
	\centering
	\includegraphics[scale=.4]{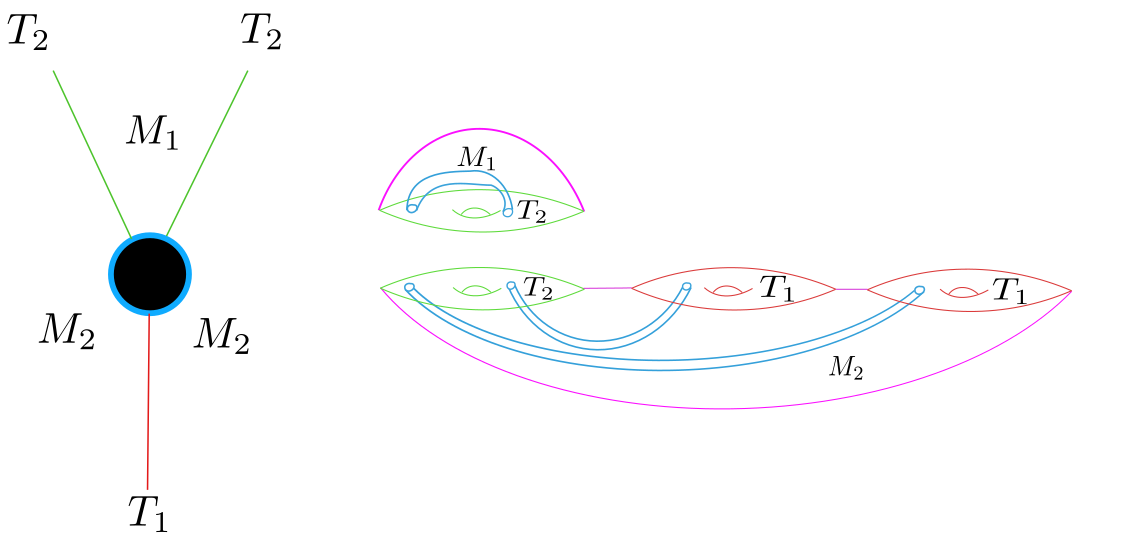}
	\caption{On the right side, the combination of the green and blue parts on the top combine to form $S_1$, and the combination of the green, red, and blue parts on the bottom combine to form $S_2$. The blue annuli come from a tubular neighborhood of the knot, and these are filled in when doing 0-surgery on $M$ to form the JSJ components.}
\end{figure} \\
We have the restriction map $r: X(M) \to X(M_1) \times X(M_2)$. We know that $\dim(X(M_1)) \geq 3$ and $\dim(X(M_2)) \geq 6$, so $\dim(X(M_1) \times X(M_2)) \geq 9$. Let $V$ be the variety inside $X(M_1) \times X(M_2)$ defined by the gluing maps, $\varphi_1$ between once-punctured tori, $\varphi_2$ between twice-punctured tori. Then $\varphi_1$ gives 2 equations (as in the proof of the main theorem from \cite{Wang}), while $\varphi_2$ gives five equations (as in the proof of Theorem \ref{thm:main}). There is also an equation which equates the traces of all punctures. So $V$ is defined by 8 equations in a variety whose dimension is at least 9 and is hence is least 1-dimensional inside $X(M_1) \times X(M_2)$. SnapPy and Sage compute that $X(M)$ is 1-dimensional. Thus $V$ is 1-dimensional by the same argument as the proof of Theorem \ref{thm:main}. Let $\mathcal{O}_1$ be the trefoil knot complement (i.e. $M_1$ with a 2-handle glued to the boundary) and let $\mathcal{O}_2$ be the trivial Seifert fibering over the thrice-punctured sphere (i.e. $M_2$ with two 2-handles glued to the boundary). Let $\chi_1: \pi_1(M_1) \to \mathbb{C}, \chi_2: \pi_1(M_2) \to \mathbb{C}$ be the holonomy traces of $\rho_1, \rho_2$ from Lemma \ref{lma:pretzels}, which realize $SL_2(\mathbb{C})$-compatibility. Then $\chi_\infty = (\chi_1, \chi_2) \in V$. Then $\chi_\infty$ is reducible when restricted to the boundaries of $M_1, M_2$.

\medskip

We now argue that it is not possible for $(\chi_1, \chi_2) = \chi_\infty \in V$ to be surrounded by traces of representations that are reducible on $S_1$ (i.e. the boundary component containing a glued twice-punctured torus $T_2'$). By definition, the traces of the two punctures of $T_2'$ must be equal in $V$. Suppose we have a sequence $\{\chi_j'\}_{i=1}^\infty \subset V$ approaching $\chi_\infty$. Let $\gamma$ be the product of the two punctures in $T_2'$. If $\chi_j'(\gamma)$ is not constantly equal to 2 on this sequence, then since $\gamma$ is a commutator on $T_2'$, $\chi_j'$ consists of irreducible traces. If $\chi_j'$ stays constantly equal to 2 on $\gamma$, then if the trace of the puncture on $T_2'$ is not constantly equal to 2 on this sequence, the commutator of the two punctures does not have constant trace 2, and $\chi_j'$ thus consists of irreducible representations. Thus, the only way for $\chi_j'$ to stay reducible on $\pi_1(T_2')$ is for $\chi_j'(g) = \chi_1(g)$ constantly for all $g \in T_2'$. If this were the case, then each $\chi_j'$ must project to the trace of some boundary-parabolic representation on the trefoil knot complement and the trivial Seifert fibering over the thrice-punctured sphere; then there would be a positive-dimensional locus of boundary-parabolic reperesentations on the trefoil knot complement or the trivial Seifert fibering over the thrice-punctured sphere. However, it is known that these manifolds do not admit positive-dimensional spaces of boundary-parabolic representations, creating a contradiction. Thus, there exists some sequence $\chi_j'$ of traces projecting to irreducible traces on $\pi_1(M_1), \pi_1(M_2)$ approaching $\chi_\infty$. By Lemma 6 in \cite{Paoluzzi2010ConwaySA}, each $\chi_j'$ lifts to some sequence $\chi_j'' \subset X(H)$ contained in either $C_1, C_2$, or $C_3$, such that the trace of $\ell_1$ approaches 2, since $\chi_\infty(\ell_1) = 2$. However, since $\chi_1$ and $\chi_2$ do not glue to form a representation on $\pi_1(H)$, $\chi_j''$ must approach an ideal point of $C_1, C_2$, or $C_3$. Sage computes that for $C_1$, the trace of $\ell_1$ is constantly equal to -2, and for $C_2$, the trace of $\ell_1$ is constantly equal to 0. Thus, $\chi_j''$ must approach an ideal point of $C_3$. On $C_3$, the trace of $[m_1, \ell_1]$ is not constantly equal to 2, so by Lemma 6 of \cite{Paoluzzi2010ConwaySA}, there exists some $\chi_j''' \subset X(M)$ restricting to $\chi_j'' \subset C_3 \subset X(H)$, which in turn restricts to $\chi_j' \subset V \subset X(M_1) \times X(M_2)$, converging to $\chi_\infty \in V \subset X(M_1) \times X(M_2)$. Since the holonomy representations $\rho_1, \rho_2$ are never conjugate on the glued JSJ torus boundaries, we are in the situation of Lemma 7 of \cite{Paoluzzi2010ConwaySA}, and hence $\mathbb{T}$ is detected by an ideal point on the component of $X(M)$ corresponding to $r'^{-1}(C_3)$, with limiting character restricting to $\chi_1$ and $\chi_2$. 
\end{proof}

\begin{remark}
Experimenting with Sage, one can find that $r'^{-1}(C_3)$ is actually a canonical component inside $X(M)$, so the ideal point detecting $\mathbb{T}$ is actually an ideal point of a canonical component. 
\end{remark}

\begin{remark}
There is another pair of lifts of holonomy representations realizing $SL_2(\mathbb{C})$-compatibility for the 0-surgery whose peripheral traces match the constant traces on the component $C_1 \subset \mathbb{C}^3$. It is possible that $C_1$ corresponds to a non-norm curve $D \subset X(M)$ coming from a component of $X(M(0))$, since the trace of the longitude is constantly equal to 2 on this component, and that the pair of surfaces is detected by an ideal point on $D$. The techniques in this paper are not well-equipped to deal with this scenario, since we cannot glue reducible characters on once-punctured tori in $X(H)$ to form characters in $X(M)$. 
\end{remark}

\bibliography{ExtensionsSequel.bib}

\begin{thebibliography}{10}

\bibitem{Aschenbrenner20123ManifoldG}
Matthias Aschenbrenner, Stefan Friedl, and Henry Wilton.
\newblock 3-manifold groups.
\newblock {\em EMS Lecture Series in Mathematics}, 2015.

\bibitem{Chinburg2017AzumayaAA}
Ted Chinburg, Alan Reid, and Matthew Stover.
\newblock Azumaya algebras and canonical components.
\newblock {\em International Mathematics Research Notices}, 2022, 06 2017.

\bibitem{ClayTeragaito}
Adam Clay and Masakazu Teragaito.
\newblock Left-orderability and exceptional {D}ehn surgery on two-bridge knots.
\newblock In {\em Geometry and topology down under}, volume 597 of {\em
  Contemp. Math.}, pages 225--233. Amer. Math. Soc., Providence, RI, 2013.

\bibitem{Cooper1994PlaneCA}
Daryl Cooper, Marc Culler, Henri Gillet, Darren~D. Long, and P.~B. Shalen.
\newblock Plane curves associated to character varieties of 3-manifolds.
\newblock {\em Inventiones mathematicae}, 118:47--84, 1994.

\bibitem{Culler1986LiftingRT}
Marc Culler.
\newblock Lifting representations to covering groups.
\newblock {\em Advances in Mathematics}, 59:64--70, 1986.

\bibitem{Culler1983VarietiesOG}
Marc Culler and P.~B. Shalen.
\newblock Varieties of group representations and splittings of 3-manifolds.
\newblock {\em Annals of Mathematics}, 117:109--146, 1983.

\bibitem{Dunfield2008AP1}
Nathan~M. Dunfield.
\newblock Examples of non-trivial roots of unity at ideal points of hyperbolic
  3-manifolds.
\newblock {\em Topology}, 38:457--465, 1998.

\bibitem{eudavemunoz}
Mario Eudave-Mu\~noz.
\newblock Non-hyperbolic manifolds obtained by {D}ehn surgery on hyperbolic
  knots.
\newblock In {\em Geometric topology ({A}thens, {GA}, 1993)}, volume 2.1 of
  {\em AMS/IP Stud. Adv. Math.}, pages 35--61. Amer. Math. Soc., Providence,
  RI, 1997.

\bibitem{Goldman}
William Goldman.
\newblock Trace coordinates on fricke spaces of some simple hyperbolic
  surfaces.
\newblock {\em Handbook Handbook of Teichmüller theory. Vol. II, IRMA Lect.
  Math. Theor. Phys}, 13, 01 2009.

\bibitem{gordonleucke}
C.~McA. Gordon and John Luecke.
\newblock Non-integral toroidal {D}ehn surgeries.
\newblock {\em Comm. Anal. Geom.}, 12(1-2):417--485, 2004.

\bibitem{Gordon2000DehnSO}
Cameron~McA. Gordon and John~S Luecke.
\newblock Dehn surgeries on knots creating essential tori, ii.
\newblock {\em Communications in Analysis and Geometry}, 8:671--725, 2000.

\bibitem{hanselman}
Jonathan Hanselman, Jacob Rasmussen, Sarah~Dean Rasmussen, and Liam Watson.
\newblock L-spaces, taut foliations, and graph manifolds.
\newblock {\em Compos. Math.}, 156(3):604--612, 2020.

\bibitem{hatcheroertel}
A.~Hatcher and U.~Oertel.
\newblock Boundary slopes for {M}ontesinos knots.
\newblock {\em Topology}, 28(4):453--480, 1989.

\bibitem{Hatcher1985IncompressibleSI}
Allen Hatcher and William~P. Thurston.
\newblock Incompressible surfaces in 2-bridge knot complements.
\newblock {\em Inventiones mathematicae}, 79:225--246, 1985.

\bibitem{IchiharaJongKabaya}
Kazuhiro Ichihara, In~Jong, and Yuichi Kabaya.
\newblock Exceptional surgeries on {$(-2, p, p)$}-pretzel knots.
\newblock {\em Topology and its Applications}, 159:1064–1073, 03 2012.

\bibitem{IchiharaMasai}
Kazuhiro Ichihara and Hidetoshi Masai.
\newblock Exceptional surgeries on alternating knots.
\newblock {\em Communications in Analysis and Geometry}, 24, 10 2013.

\bibitem{Kitano1994ReidemeisterTO}
Teruaki Kitano.
\newblock Reidemeister torsion of seifert fibered spaces for
  {$SL(2;\mathbb{C})$}-representations.
\newblock {\em Tokyo Journal of Mathematics}, 17:59--75, 1994.

\bibitem{LidmanWatson}
Tye Lidman and Liam Watson.
\newblock Nonfibered {L}-space knots.
\newblock {\em Pacific J. Math.}, 267(2):423--429, 2014.

\bibitem{Macasieb2009OnCV}
Melissa~L. Macasieb, Kathleen~L. Petersen, and Ronald van Luijk.
\newblock On character varieties of two‐bridge knot groups.
\newblock {\em Proceedings of the London Mathematical Society}, 103, 2009.

\bibitem{Mattman}
Thomas~W. Mattman.
\newblock The {C}uller-{S}halen seminorms of the {$(-2,3,n)$} pretzel knot.
\newblock {\em J. Knot Theory Ramifications}, 11(8):1251--1289, 2002.

\bibitem{MenalFerrer2010TwistedCF}
Pere Menal-Ferrer and Joan Porti.
\newblock Twisted cohomology for hyperbolic three manifolds.
\newblock {\em Osaka Journal of Mathematics}, 49, 01 2010.

\bibitem{nizhang}
Yi~Ni and Xingru Zhang.
\newblock Detection of knots and a cabling formula for {$A$}-polynomials.
\newblock {\em Algebr. Geom. Topol.}, 17(1):65--109, 2017.

\bibitem{ohtsuki}
Tomotada Ohtsuki.
\newblock Ideal points and incompressible surfaces in two-bridge knot
  complements.
\newblock {\em J. Math. Soc. Japan}, 46(1):51--87, 1994.

\bibitem{Paoluzzi2010ConwaySA}
Luisa Paoluzzi and Joan Porti.
\newblock Conway spheres as ideal points of the character variety.
\newblock {\em Mathematische Annalen}, 354:707--726, 2010.

\bibitem{Petersen2014GonalityAG}
Kathleen~L. Petersen and Alan~W. Reid.
\newblock Gonality and genus of canonical components of character varieties.
\newblock {\em Characters in Low-Dimensional Topology}, 2014.

\bibitem{Schleimer2018IntroductionTT}
Saul Schleimer.
\newblock Introduction to towers in topology.
\newblock 2018.

\bibitem{sekino}
Nozomu Sekino.
\newblock The {JSJ}-decomposition of the 3-manifold obtained by 0-surgery along
  a classical pretzel knot of genus one.
\newblock {\em preprint}, 2024.

\bibitem{Teragaito}
Masakazu Teragaito.
\newblock Left-orderability and exceptional dehn surgery on twist knots.
\newblock {\em Canadian Mathematical Bulletin}, 56, 2013.

\bibitem{Thurston1979TheGA}
W.P. Thurston.
\newblock The geometry and topology of three-manifolds.
\newblock Princeton University, 1979.

\bibitem{tillmann}
Stephan Tillmann.
\newblock Character varieties of mutative 3-manifolds.
\newblock {\em Algebr. Geom. Topol.}, 4:133--149, 2004.

\bibitem{orderability}
Yi~Wang.
\newblock Detected {S}eifert surfaces and intervals of left-orderable
  surgeries.
\newblock {\em preprint}, 2025.

\bibitem{Wang}
Yi~Wang.
\newblock Punctured {JSJ} tori and tautological extensions of {A}zumaya
  algebras.
\newblock {\em Algebraic and Geometric Topology}, to appear.

\end{thebibliography}
\bibliographystyle{plain}

\end{document}